\documentclass[12pt]{article}
\usepackage[utf8]{inputenc}
\usepackage{geometry}
\usepackage{amsfonts,amsmath,amssymb,amsthm}

\usepackage{graphicx}

\usepackage{xcolor}
\usepackage{soul}

\usepackage{soul}

\newtheorem{theorem}{Theorem}
\newtheorem{lemma}{Lemma}
\newtheorem{remark}{Remark}
\newtheorem{corollary}{Corollary}

\title{On the number of heterochromatic trees in nice and beautiful colorings of complete graphs\thanks{Partially supported by CONACyT, M\'exico, and PAPIIT, UNAM.}}
\author{Juan Jos\'e Montellano-Ballesteros\thanks{Instituto de Matem\'aticas, Universidad Nacional Aut\'onoma de M\'exico} \and Eduardo Rivera-Campo\thanks{Departamento de Matem\'aticas, Universidad Aut\'onoma Metropolitana-I} \and Ricardo Strausz\thanks{Instituto de Matem\'aticas, Universidad Nacional Aut\'onoma de M\'exico} }

\date{}

\begin{document}

\maketitle

\abstract{We introduce classes of  edge-colourings of the complete graph
 --- that we call nice and beautiful --- and study  how many heterochromatic spanning trees appear under such colourings. We prove that if the colouring is nice, there is at least a quadratic number of different heterochromatic trees; and if the colouring  is beautiful there is an exponential number of different such trees.}

\section{Introduction}

Let $G=(V,E)$ be a simple graph of order $n+1$ and size $m \geq n$. Consider an injective labeling $\gamma:V\to\{0,\dots,m\}$ of its vertices and assign to each edge $e=uv$ the number $\gamma'(e)=|\gamma(u)-\gamma(v)|$. If every edge is assigned to a different number, we say that $\gamma$ is a {\it graceful labeling\/}, and if such a labeling exists we call $G$ a {\it graceful graph\/}. The long-standing Graceful Tree Conjecture ---also known as Ringel-Kotzig-Rosa conjecture \cite{Rin63, K73, Ro67}, or RKR-conjecture for short--- says that all trees are graceful graphs.

An $m$-edge-colouring of a graph $G$ is an assigment  of colours to the edges of $G$ that uses $m$ colours. Equivalently, an $m$-edge-colouring of a graph $G$ is a partition $E=C_1\sqcup C_2\sqcup \ldots\sqcup C_m$ of the edge set of $G$; the sets $C_1, C_2, \ldots, C_m$ are called the {\it colour classes}.  A subgraph $H$ of an edge-coloured graph $G$ is {\it heterochromatic} if all edges in $H$ have different colours.

If we label the vertices of the complete graph $K_{n+1}$ with the numbers $\{0,1,\dots,n\}$ and assign to each edge the absolute value of the difference between the labels of its vertices, we get a colouring of the edges with $n$ colours which we call {\it the graceful colouring\/} of the complete graph ---we refer to the numbers in $[n]=\{1,2,\dots,n\}$ as colours when assigned to the edges of the complete graph. In terms of this colouring, the RKR-conjecture says that there is an {\it heterochromatic\/} copy of each tree in the graceful colouring of the complete graph; that is, for each tree of order $n+1$, there is an isomorphic copy of it, inside the gracefully coloured $K_{n+1}$, which uses $n$ different colours in its edges. This ``beautiful'' colouring has several properties which we use along the paper to prove the existence of ``many'' heterochromatic trees.

To begin with, observe that in the graceful colouring the colour $c\in[n]$ is used exactly $n-c+1$ times; 
so the $n$  colour classes have sizes $n, n-1, \ldots, 2, 1$, respectively. An $n$-colouring of the edges of the complete graph $K_{n+1}$ with such a property is called a {\it nice\/} colouring. For example, consider the following recursive ``Stellar'' colouring of a complete graph: start with a vertex; then, at each step, add a new vertex and colour all edges joining it to the previously added vertices with the same and new colour --- that is, at each step, add a monochromatic star, see Fig. \ref{stellar}. Clearly this is a nice colouring; furthermore, it is easy to see that this nice colouring has a heterochromatic copy of each tree of size $n$.

\begin{figure}[h!]
\centering
  \includegraphics[width= 2in]{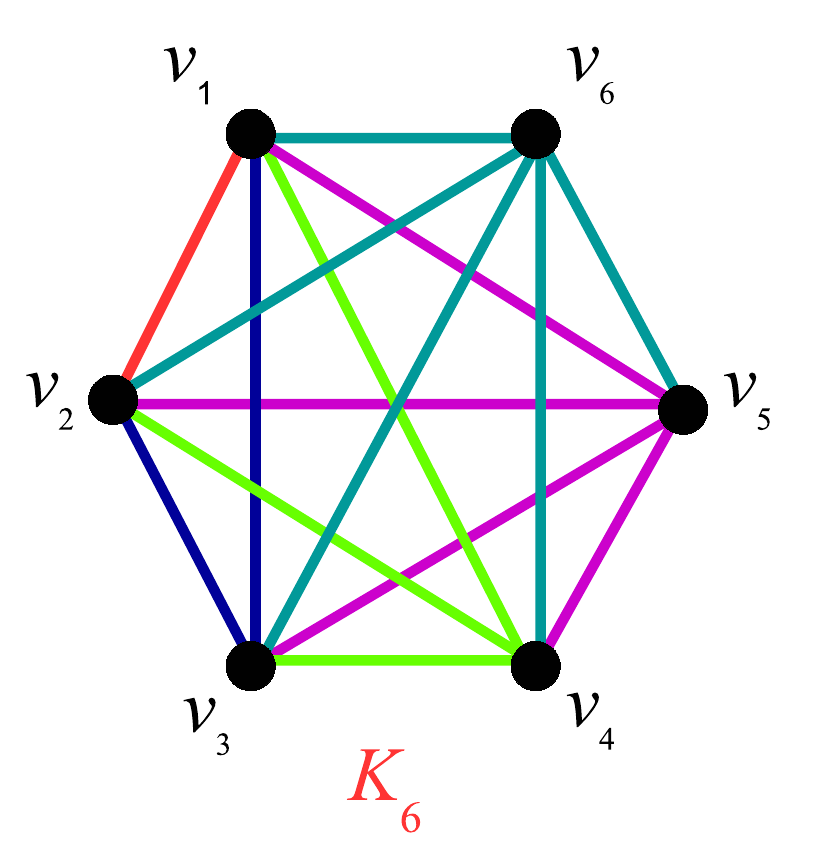}
  \caption{Stellar colouring of $K_6$.}
  \label{stellar}
\end{figure}

The graceful and stellar colourings of the complete graph  $K_{n+1}$ share many common properties: in both cases all monochromatic subgraphs are acyclic and $K_{n+1}$  can be decomposed into heterochromatic stars to mention just a couple. In Section 3 we show for every nice edge colouring of the complete graph $K_{n+1}$, there are $\Omega(n^2)$ different heterochromatic spanning trees. 

Later, in Section \ref{sectionbeautiful} we define the class of {\it beautiful\/} edge-colourings of $K_{n+1}$, which includes the class of nice colourings, and prove that for any such a colouring, the graph $K_{n+1}$ has $\Omega(2^n)$ different heterochromatic spanning trees. 

The existence of heterochromatic trees in edge colourings has been studied by various authors. {In particular  K. Suzuki  \cite{Su06} and S. Akbari and A. Alipour \cite{AkAl06} characterise independently those edge colourings that contain heterochromatic trees; viz.

\medskip
\noindent{\bf Theorem A (K. Suzuki). } {\it Let $C$ be an edge colouring of a graph $G$ with $n$ vertices. There is a heterochromatic spanning tree of $G$ if and only if for any set of $r$ colours of $C$ with $1\leq r \leq n-2$, the graph obtained from $G$ by removing all edges of $G$ coloured with any of these $r$ colours has at most $r+1$ connected components.
}\medskip

\medskip
\noindent{\bf Theorem B (S. Akbari, A. Alipour). } {\it Let $G$ be an edge-coloured graph. Then $G$ has a heterochromatic spanning tree if and only if for every partition of $V(G)$ into $t$ parts, with $1 \leq t \leq \vert V(G) \vert$, there are at least $t-1$ edges with distinct colours whose ends lie in different parts.
}\medskip

Other results concerning the complete graph have been given too. For example, colourings given by perfect matchings of the complete graph $K_{2n}$ were studied by R. A. Brualdi and S. Hollingsworth \cite{BH96}. 

$\bullet\quad$ {\it If $C$ is an edge colouring of the complete graph $K_{2n}$ given by $2(n-1)$ disjoint perfect matchings, then $K_{2n}$ contains two edge-disjoint heterochromatic trees.}

An upper bound for the minimum integer $k$ such that for every edge $k$-colouring of the complete graph $K_n$ there exists a heterochromatic spanning tree, was found by A. Bialostocki and W. Voxman \cite{BV01}.

$\bullet\quad$ {\it If $C$ is an edge colouring of the complete graph $K_n$ with at least $\binom{n-2}{2} +2$ colours, then $K_n$  has a heterochromatic} spanning tree.

This last proposition can be generalised in various directions. For examples, J. Arocha and V. Neumann \cite{AN} generalised Bialostocki's result for arbitrary graphs and J. J. Montellano-Ballesteros and E. Rivera-Campo \cite{MR13} gave the corresponding result for matroids.

$\bullet\quad$ {\it Let $G$ be a simple connected graph with with $m \geq 2$ edges. If $C$ is an edge colouring of $G$  with exactly $m - \tau (G) + 2$ colours, then $G$ has  a heterochromatic spanning tree.}

$\bullet\quad$ {\it Let $M$ be a matroid with $m$ elements and rank at least 2. If $C$ is a colouring of the elements of $M$ with at least $m - \tau (M) + 2$ colours, then $G$ has a heterochromatic basis.}

Were, $\tau (G)$ (respectively $\tau (M)$) denote the size of the smallest set of edges of $G$ (elements of $M$) which contains at least two edges (elements) of each spanning tree of $G$ (each basis of $M$).

\section{Preliminary results}

Let $n \geq 0$ be an integer and $G$ be a graph with $n+1$ vertices. For an $n$-edge-colouring of $G$ let $M_1 = (E(G), \mathcal{I}_1)$ and $M_2 = (E(G), \mathcal{I}_2)$ be matroids with ground set $E(G)$ and independence sets $\mathcal{I}_1$ and $\mathcal{I}_2$, respectively, where $X\in \mathcal{I}_1$ if the subgraph $G[X]$ of $G$ induced by $X$ is acyclic and $X\in \mathcal{I}_2$ if $G[X]$ is heterochromatic. A common independent set in $M_1$ and $M_2$ is the edge set of a heterochromatic forest of $G$.

For $X\subset E(G)$ we denote by $w(X)$ and $c(X)$, respectively,  the number of connected components of the spanning subgraph of $G$ with edge set $X$ and the number of colour classes contained in $X$. 

The following lemma will be used in the proofs of our main results. A proof of the lemma can be obtained using Suzuki's Theorem but, for the sake of completeness, we present an alternative proof that uses Edmond's Matroid Intersection Theorem \cite{Ed70}.

\begin{lemma}
\label{lemabasico}
Let $C$ be an $n$-edge-colouring of a graph $G$ with $n+1$ vertices. If $w(X) + c(X) \leq n+1$ for all $X \subset E(G)$, then $G$ has a heterochromatic spanning tree.
\end{lemma}

\begin{proof}
Let $X \subset E(G)$. Then $r_{M_1}(X) = n +1 - w(X)$ and $r_{M_2}(E(G) \setminus X) = n - c(X)$. Therefore
\begin{align*}
    r_{M_1}(X) + r_{M_2}(E(G) \setminus X) &= (n+1-w(X)) + (n - c(X))\\
    &=(2n+1) - (w(X) + c(X)) \\
    &\geq (2n+1) - (n+1)\\
    &= n.
\end{align*}

By Edmonds' Matroid Intersection Theorem , $E(G)$ contains a set $X^*$ with size $n$ which is independent in both matroids $M_1$ and $M_2$. This implies that the subgraph of $G$ induced by $X^*$ is a heterochromatic spanning tree of $G$.

\end{proof}

\section{Heterochromatic spanning trees in cute and nice colourings of graphs}
\label{sectionniceandcute}

An $n$-edge-colouring of the complete graph $K_{n+1}$ is a \emph{nice colouring} if the $n$  colour classes have sizes $1, 2, \ldots, n$. Let $G$ be a graph with $n+1$ vertices and $1 + {n \choose 2}$ edges. An $n$-edge-colouring of $G$ is a \emph{cute} edge-colouring of $G$ if the sizes of the $n$  colour classes are $1, 1, 2, \ldots, n-1$.

With Lemma \ref{lemabasico} in hand, we prove the following theorems concerning cute and nice colourings.

\begin{theorem}
\label{cute} 
Let $G$ be a graph with $n+1$ vertices and $1 + {n \choose 2}$ edges. 
If $C$ is a cute $n$-edge-colouring of $G$, then $G$ has a heterochromatic spanning tree.
\end{theorem}

\begin{proof}
Let $C_1, C_2, \ldots, C_{n}$ be the  colour classes of $C$. Without loss of generality we assume $\vert C_1 \vert = 1$ and $\vert C_i \vert = i-1$ for $i=2, 3, \ldots n$.
Let $X \subset E(G)$ and denote by $E_1, E_2, \ldots, E_{w(X)}$ the sets of edges of the connected components of the spanning subgraph of $G$ with edge set $X$. Then $$\vert X \vert = \sum_{i=1}^{w(X)} \vert E_i \vert \leq  {n+2 -w(X) \choose 2}$$ since the number of edges of a graph is maximum when all edges lie in one connected component.
    
On the other hand if $C_{i_1}, C_{i_2}, \ldots, C_{i_{c(X)}}$ are the  colour classes contained in $X$, then
\begin{align*}
    \vert X \vert &= \sum_{i=1}^{n} \vert X \cap C_{i} \vert\\
    &\geq \sum_{j=1}^{c(X)} \vert C_{i_j} \vert\\
    &\geq 1 + 1+ 2+ \ldots + (c(X) - 1)\\
    &= 1 + {c(X) \choose 2}.
\end{align*}
Therefore 
$${n+2-w(X) \choose 2} \geq 1 + {c(X) \choose 2}$$
This implies $${n+2-w(X) \choose 2} > { c(X) \choose 2}$$ which in turn gives $n+2-w(X) > c(X)$ and therefore $w(X) + c(X) \leq n+1$. By Lemma \ref{lemabasico}, $G$ has a heterochromatic spanning tree.

\end{proof}

The following remark shows that the condition in Theorem \ref{cute} can only guarantee the existence of one heterochromatic spanning tree, see Fig. \ref{uniquetree}.

\begin{figure}[h!]
\centering
  \includegraphics[width= 4in]{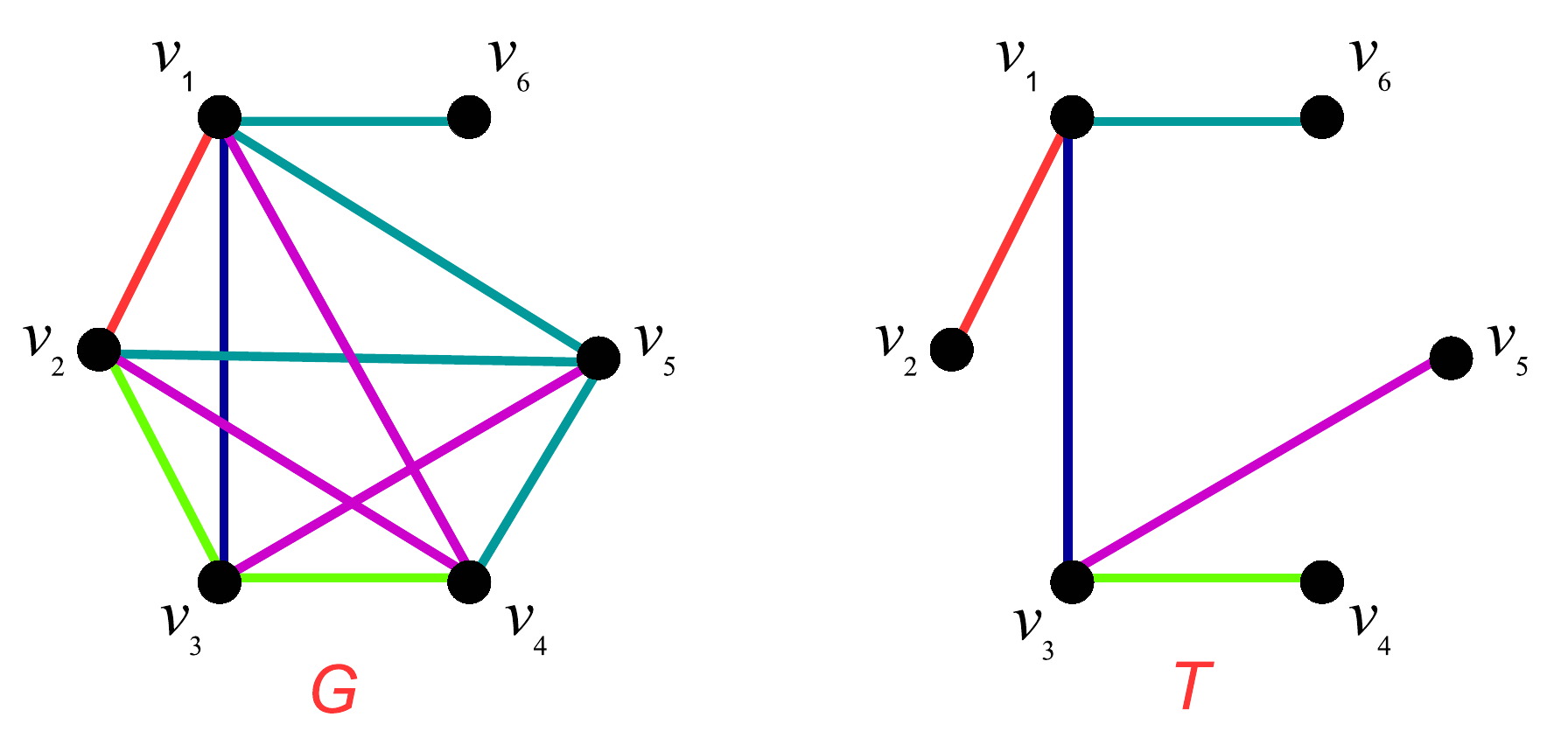}
  \caption{Cute colouring of a graph $G$ with exactly one heterochromatic spanning tree $T$.}
  \label{uniquetree}
\end{figure}

\begin{remark}
\label{remarkunique}
For each tree $T$ with $n+1$ vertices, there is a spanning supergraph $G(T)$ of $T$ with  $1 + \binom{n}{2}$ edges and a cute edge-coloring of $G(T)$ for which $T$ is the unique  heterochromatic spanning tree.

\end{remark}

\begin{proof}

Let $T$ be a tree with vertices $v_1, v_2, \ldots, v_{n+1}$. Without loss of generality assume that for $i=1, 2, \ldots, n+1$, the subgraph $T_i$ of $T$ induced by $v_1, v_2, \ldots, v_i$ is a tree. For $i=2, 3, \ldots, n+1$ let $v_i^-$ be the unique vertex of $T_{i-1}$ adjacent to $v_i$ in $T_i$. Then $E(T) = \{v_i^-v_i : i = 2, 3, \ldots, n+1\}$.

Let $G(T)$ be the supergraph of $T$ with edge set 
$$E(G(T)) =  \{v_iv_j: 1\leq i \leq j \leq n\} \cup \{v_{n+1}^-v_{n+1}\};$$ 
clearly $G(T)$ is a spanning supergraph of $T$ with $1 + \binom{n}{2}$ edges. 

Let $C$ be the edge-colouring of $G(T)$ with  colour classes $C_1, C_2, \ldots, C_{n}$ given by: $C_1 = \{v_1v_2\} = \{v_2^-v_2\}$ and for $k=2,3,  \ldots, n$,  $C_{k} = \{v_{k+1}^-v_{k+1}\} \cup \{v_iv_k: i=1, 2, \ldots k-1, v_i \neq v_k^-\}$. Notice that  $C$ is a cute edge-colouring of $G(T)$ since $\vert C_1\vert = 1$ and $\vert C_k \vert = 1 + (k-2) =k-1$ for $k=2, 3, \ldots, n$. We claim that $T$ is the only spanning tree of $G(T)$ which is heterocromatic.

Let $H$ be a heterocromatic spanning tree of $G(T)$. Edges $v_2^-v_2$ and $v_3^-v_3$ are the  unique edges in $C_1$ and $C_2$, respectively, therefore they must be edges of $H$. This implies that $T_1$ and $T_2$ are subtrees of $H$. Assume $T_k$ is a subtree of $H$; since $C_{k} = \{v_{k+1}^-v_{k+1}\} \cup \{v_iv_k: i=1, 2, \ldots k-1, v_i \neq v_k^-\}$ and all edges $v_iv_k: i=1, 2, \ldots k-1$ have both ends in $T_k$, the only possible edge in $C_k$ that lies in $H$ is edge $v_{k+1}^-v_{k+1}$. Therefore $T_{k+1}$ is a subtree of $H$. This inductive argument shows that $T = T_{n+1}$ is a subtree of $H$ which imples $H = T$.

\end{proof}

Unlike the stellar colouring, not every nice colouring of $K_n$ contains a heterochromatic copy of every tree witn $n$ vertices, see Fig. \ref{nostars}. Nevertheless every nice colouring of $K_n$ produces many different heterochromatic spanning trees.

\begin{figure}[h!]
\centering
  \includegraphics[width= 2in]{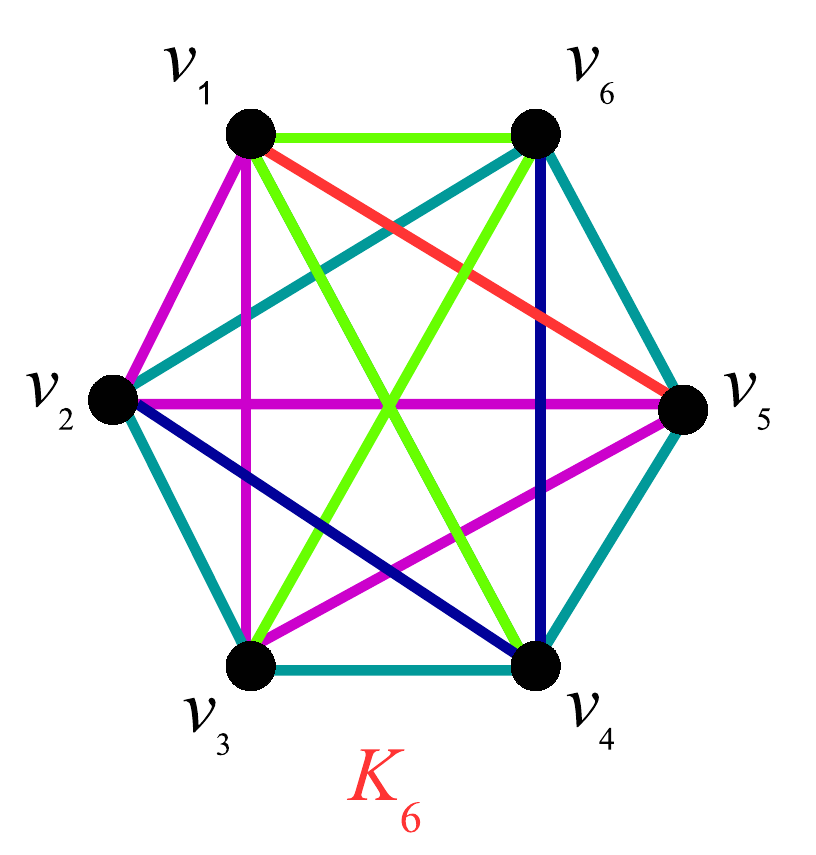}
  \caption{A nice colouring of $K_6$ with no heterochromatic spanning tree $T$ with $\Delta (T) \geq 4$.}
  \label{nostars}
\end{figure}

\begin{theorem}
\label{nice}
Let $G$ be a complete graph with $n+1$ vertices. If $C$ is a nice $n$-edge-colouring of $G$, then $G$ has at least $\lceil \frac{n+1}{2} \rceil \lfloor \frac{n+1}{2}  \rfloor$ different heterochromatic spanning trees.
\end{theorem}

\begin{proof}
Let $C_1, C_2, \ldots, C_{n}$ be the colour classes of $C$. Without loss of generality assume $\vert C_i \vert  = i$ for $i=1, 2, \ldots n$.

Let $e_1$ be an edge in $C_{\lceil \frac{n+1}{2} \rceil}$. We claim $G$ has at least $\lfloor \frac{n+1}{2} \rfloor $ different heterochromatic spanning trees containing the edge $e_1$. 

First choose an arbitrary subset $W_1$ of $C_{n}$ with size $\lceil \frac{n+1}{2} \rceil$ and let $G_{1,1}$ be the subgraph of $G$ with edge set 
$$E(G_{1,1}) = \left(E(G) \setminus (C_{\lceil \frac{n+1}{2} \rceil} \cup C_{n})\right) \cup (\{e_1\} \cup W_1)$$ 

When restricted to the graph $G_{1,1}$ colouring $C$ is a cute edge-colouring  since $\vert C_1 \vert = \vert \{e_1\} \vert = 1$, $\vert C_i \vert = i $ for $i \neq C_{\lceil \frac{n+1}{2} \rceil}$ and $\vert W_1 \vert = \lceil \frac{n+1}{2} \rceil$. By Theorem \ref{cute}, $G_{1,1}$ has a heterocromatic spanning tree $T_{1,1}$.

Assume $T_{1,1}, T_{1,2}, \ldots , T_{1,t}$ are different heterocromatic spanning trees of $G$ containing edge $e_1$. For $i = 1, 2, \ldots, t$ let $f_i$ be the edge in $T_i$ having colour $n$. If $t < \lfloor \frac{n+1}{2}  \rfloor$, then we can choose $W_{t+1} \subset C_{n}$ with size $\lceil \frac{n+1}{2} \rceil$ containing none of the edges $f_1, f_2, \ldots, f_t$. Let $G_{1,t+1}$ be the subgraph of $G$ with edge set 

$$E(G_{1,t+1}) = \left(E(G) \setminus (C_{\lceil \frac{n+1}{2} \rceil} \cup C_{n})\right) \cup (\{e_1\} \cup W_{t+1})$$ 

As with the case of $G_{1,1}$, when restricted to $G_{1,t+1}$, colouring $C$ is a cute edge-colouring. By Theorem \ref{cute}, $G_{1, t+1}$ has a heterochromatic tree $T_{1,t+1}$. Notice that $T_{1,t+1} \neq T_{1,i}$ for $i=1, 2, \ldots, t$ since $f_i$ is an edge of $T_{1,i}$ and not an edge of $T_{1, t+1}$.

To end the proof, repeat the previous argument for each edge $e_2, e_3, \ldots, e_{_{\lceil \frac{n+1}{2} \rceil}} \in C_{\lceil \frac{n+1}{2} \rceil}$ obtaining different heterochromatic trees $T_{i,j}$ of $G$ with $i=1,2, \ldots, \lceil \frac{n+1}{2} \rceil$ and $j=1, 2, \ldots, \lfloor \frac{n+1}{2}  \rfloor$.
\end{proof}

\section{Heterochromatic spanning trees in beautiful colourings of complete graphs}
\label{sectionbeautiful}

Let $C$ be a nice $n$-edge-colouring of $K_{n+1}$,  $\{C_1, C_2, \dots, C_n\}$ be its colour classes,  where for each $i=1, \dots, n$, $|C_i| =i$. We  denote by $G_i$ the subgraph of $K_{n+1}$ induced by $C_i$.

The edge-coloring $C$ will be called  {\it beautiful} if for  every colour class  $C_{i}$ we have that $G_i$ is acyclic and  there is a partition $\{V_1, V_2\}$ of $V(K_{n+1})$, with $|V_2|= \lceil\frac{n+1}{2}\rceil \geq \lfloor\frac{n+1}{2}\rfloor = |V_1|$,  such that:

\noindent  {\it i)} The subgraph of $K_{n+1}$ induced by all the  colour classes $C_i$, with $i \equiv n$ (mod 2), is isomorphic to the complete bipartite graph with parts $\{V_1, V_2\}$.

\noindent  {\it ii)} For every  colour class $C_{i}$, with $i \equiv n$ (mod 2), we have that  $|V(G_i) \cap V_1| =  \lfloor\frac{|V(G_i)|}{2}\rfloor$ and $|V(G_i) \cap V_2|= \lceil\frac{|V(G_i)|}{2}\rceil$.

\noindent {\it iii)} For every colour class  $C_{i}$,  with $i \not\equiv n$ (mod 2),  we have that 
$|C_i\cap E(K_{n+1}[V_1]) | = \lfloor\frac{i}{2}\rfloor$ and  $|C_{i}\cap E(K_{n+1}[V_2]) |= \lceil\frac{i}{2}\rceil$.

\begin{theorem}\label{beautiful} Let $n\geq 2 $ be an integer and $C$ be a beautiful  $n$-edge coloring of $K_{n+1}$. 
Then in $K_{n+1}$ there are $2^{\lfloor\frac{n-1}{2}\rfloor}$ heterochromatic spanning trees.
\end{theorem}
\begin{proof} The proof will be done as follows. Consider a spanning subgraph $G$ of $K_{n+1}$ obtained in the following way:  

\noindent {\it 1.-}  For each  colour class $C_i$,  with $i\geq 2$ and $i \not\equiv n$ (mod 2), choose one part  $V_{j_i}$ of the partition  $\{V_1, V_2\}$ of $V(K_{n+1})$, and  let $Y_i = C_i \cap  E(K_{n+1}[V_{j_i}])$.

\noindent {\it 2.-} For $i = 1$ and for $i\geq 2$, with $i \equiv n$ (mod 2), let $Y_i = C_i$. 

\noindent {\it 3.-} Let $E(G) = \bigcup\limits_{i=1}^n Y_i$. 

  Observe that $C$ induces an $n$-edge coloring of $G$. We claim that $G$ contains an heterochromatic spanning tree. 


\noindent  Let $X$ be a set of edges of $G$ and $A$ be the set of colour classes $Y_i$,  with $i \not\equiv n$(mod 2), such that $Y_i\subseteq X$;  let $B$ be the set of  colour classes $Y_i$,  with $i \equiv n$ (mod 2) such that $Y_i\subseteq X$, and suppose that $|A|= r$ and $|B|= s$.  Observe that $c(X) = r+s$.

 Let $Y_{i_0}\in A$ and $Y_{i_1}\in B$ be of maximal  size, respectively.    Assume first $n$ is even. Then  $i_0$ is odd, and since $|A|= r$ and $Y_{i_0}$ has maximal size, we see that $i_0 \geq 2r-1$ and the size of  $Y_{i_0}$ is  at least $\lfloor\frac{2r-1}{2}\rfloor$ if $Y_{i_0}$ is contained in $V_1$, and size at least  $\lceil\frac{2r-1}{2}\rceil$ if $Y_{i_0}$ is contained in $V_2$. 
 
 Similarily,  we see that $i_1$ is even, and since $|B|= s$ and $Y_{i_1}$ has maximal size, we see that $i_1 \geq 2s$ and therefore the size of  $Y_{i_1}$ is at least $2s$. Moreover, since $G_{i_1}$ is acyclic,   $|V(G_{i_1})| \geq 2s + \omega_{1}$,  where $ \omega_{1}$ is the number of connected components of $G_{i_1}$, and therefore, by  {\it ii)}, $|V_1\cap V(G_{i_1})| = \lfloor\frac{2s+\omega_{1}}{2}\rfloor$ and $|V_2\cap V(G_{i_1})| = \lceil\frac{2s+\omega_{1}}{2}\rceil$. 
 
If $Y_{i_0}$ is contained in $V_1$, let  $\{x_1, \dots, x_{\lceil\frac{2s+\omega_{1}}{2}\rceil}\} = V_2\cap V(G_{i_1})$ and for each $x_i$, choose an edge $e_i \in Y_{i_1}$ such that $x_i$ is incident to $e_i$.  Let $H$ be the subgraph of $G$ induced by $Y_{i_0}$ and $\{e_1, \dots, e_{\lceil\frac{2s+ \omega_{1}}{2}\rceil}\}$. Graph  $H$ has size at least $\lfloor \frac{2r-1}{2}\rfloor+ \lceil\frac{2s+ \omega_{1}}{2}\rceil \geq r+s$, and since $G_{i_0}$ is acyclic and each of the vertices $\{x_1, \dots, x_{\lceil\frac{2s+ \omega_{1}}{2}\rceil}\}$ has degree 1 in $H$, it follows that $H$ is acyclic. Thus,  $|V(H)| \geq r+s+  \omega_{1}$, where $ \omega_{1}$ is the number of connected components of $H$, and therefore  $\omega(X) \leq n+1- (r+s+  \omega_{1}) +  \omega_{1} = n+1-(r+s)$. Since $c(X)= r+s$, by Lemma \ref{lemabasico} we see that $G$ contains an heterochromatic spanning tree.   

If $Y_{i_0}$ is contained in $V_2$, let  $\{y_1, \dots, y_{\lfloor\frac{2s+ \omega_{1}}{2}\rfloor}\} = V_1\cap V(G_{i_1})$ and for each $y_i$, choose an edge $e_i \in Y_{i_1}$ such that $y_i$ is incident to $e_i$.  Let $H$ be the subgraph of $G$ induced by $Y_{i_0}$ and $\{e_1, \dots, e_{\lfloor\frac{2s+ \omega_{1}}{2}\rfloor}\}$. Observe that since $Y_{i_0}$ is contained in $V_2$, $Y_{i_0}$ has size at least  $\lceil\frac{2r-1}{2}\rceil$, thus  $H$ has size at least $\lceil \frac{2r-1}{2}\rceil+ \lfloor\frac{2s+ \omega_{1}}{2}\rfloor \geq r+s$. From here,  as in the previous case, we see that $G$ contains an heterochromatic spanning tree. 

For the case when  $n$ is odd,  it follows that $i_0$ is even, and since $|A|= r$ and $Y_{i_0}$ has maximal size, we see that $i_0 \geq 2r$ and the size of  $Y_{i_0}$ is at least $r$ (either if it is contained in $V_1$ or $V_2$).   Similarily,  $i_1$ is odd, and since $|B|= s$ and $Y_{i_1}$ has maximal size, we see that $i_1 \geq 2s-1$ and therefore the size of  $Y_{i_1}$ is at least $2s-1$. Moreover, since $G_{i_1}$ is acyclic, $|V(G_{i_1})| \geq 2s-1+\omega_{1}$, where $\omega_{1}$ is the number of connected components of $G_{i_1}$, and therefore, by  {\it ii)},  $|V_1\cap V(G_{i_1})| = \lfloor\frac{2s-1+\omega_{1}}{2}\rfloor$ and $|V_2\cap V(G_{i_1})| = \lceil\frac{2s-1+\omega_{1}}{2}\rceil$.   From here, in an analogous way as in the case when $n$ is even, we see that $G$ contains an heterochromatic spanning tree and the claim follows.

Since  there are $\lfloor\frac{n-1}{2}\rfloor$ colour classes $C_i$,  with $i\geq 2$ and $i \equiv n$ (mod 2), it follows there are $2^{\lfloor\frac{n-1}{2}\rfloor}$ different ways to obtain a subgraph $G$ which, by our claim, contains an heterochromatic spanning tree. Moreover, given any pair $G_1, G_2$ of these type of  subgraphs, by construction, there is at least one  colour class $C_i$,  with $i\geq 2$ and $i \equiv n$ (mod 2), such that $\left(E(G_1)\cap C_i\right)\cap \left(E(G_2)\cap C_i\right)= \emptyset$. Thus,  the heterochromatic spanning trees in $G_1$ and in  $G_2$ are different, and from here, the result follows. 
\end{proof}

\begin{corollary}\label{final}  Let $n\geq 2 $ be an integer and $C$ be the graceful colouring of $K_{n+1}$. 
Then in $K_{n+1}$ there are $2^{\lfloor\frac{n-1}{2}\rfloor}$ heterochromatic spanning trees.
\end{corollary}
\begin{proof}  
By Theorem \ref{beautiful} we only need to show that $C$  is beautiful. Let $V(K_{n+1}) =  \{v_0, v_1, \dots, v_n\}$, and, for each $i\in \{1, 2, \dots, n\}$, let $D_i = \{v_sv_t : |t-s|=i\}$, that is, $D_i$ denotes the set of edges coloured with $i\in \{1, 2, \dots, n\}$.  Since $D_i$ contains $n+1-i$ edges, we see that $D_i = C_{n+1-i}$ and,  it is not difficult to see that,  $G_{n+1-i}$ is acyclic. 

Assume first $n$ is even. Let $V_2=\{v_0, v_2, \dots, v_n\}$ and $V_1=\{v_1, v_3, \dots, v_{n-1}\}$. It is easy to see that the subgraph induced by $D_1 \cup D_3\cup\dots \cup D_{n-1}$ (that is, the subgraph induced by the union of the colour classes $C_{n+1-i}$, with $(n+1-i)\equiv n$ (mod 2)) is the complete bipartite graph with partite sets $V_1$ and $V_2$.  Hence (i) holds.  Moreover, given $D_j$, with $j$ odd,  $V_2 \cap V(G_{n+1-j})= \{v_0, \dots, v_{n-1-j}\}\cup \{v_{1+j},\dots, v_n\}$ and  $V_1 \cap V(G_{n+1-j})= \{v_1,\dots, v_{n-j}\}\cup \{v_{j},\dots, v_{n-1}\}$. Considering the cases whenever $j\leq n-j$  or not, it is not hard to see that (ii) holds.

Finally, given an even integer $2\leq j\leq n$, it is not hard to see that $D_j$ satisfies $|D_j\cap E(K_{n+1}[V_2])| = \frac{n+2-j}{2}$ and  $|D_j\cap E(K_{n+1}[V_1])| = \frac{n-j}{2}$. Thus, for each colour class $C_{n+1-j}$, with $(n+1-j)\not\equiv n$ (mod 2), we have that  $|C_{n+1-j}\cap E(K_{n+1}[V_2])| = \frac{n+2-j}{2}$ and  $|C_{n+1-j}\cap E(K_{n+1}[V_1])| = \frac{n-j}{2}$, and  (iii) holds. 

For  the case where $n$ is odd, let $V_2=\{v_0, v_2, \dots, v_{n-1}\}$ and $V_1=\{v_1, v_3, \dots, v_{n}\}$. As in the case where $n$ is even, we can show that $D_i$ (and so $C_{n+1-i}$), with $1\leq i \leq n$, satisfay the statements (i), (ii), and (iii).  Therefore, $C$  is beautiful and the corollary follows. \end{proof}

\section{Further research; heterochromatic trees in edge-colourings of bipartite graphs}
\label{sectionbipartite}

Analogous results can be found in other classes of graphs beside the complete graph using exactly the same technics. For example, we can proceed we bipartite graphs as follows. A $(2m-1)$-edge-colouring of the complete bipartite graph $K_{m,m}$ is a \emph{nice} edge-colouring if the colour classes have sizes $1, 1, 2, 2, \ldots m-1, m-1, m$. Let $G_{m,m}$ be a spanning subgraph of $K_{m,m}$ with $1 + 2{m \choose 2}$ edges.  A $(2m-1)$-edge colouring of $G_{m,m}$ is a \emph{cute} edge-colouring if the chromatic clases have sizes $1, 1, 1,2,2,\ldots, m-1, m-1$.

With the same technics as in the previous section we can prove the following results:

\begin{theorem}
\label{bipartitecute}
If $C$ is a cute edge-colouring of a subgraph $G_{m,m}$ of $K_{m,m}$ with $1 + 2{m \choose 2}$ edges, then $G$ has a heterochromatic spanning tree.
\end{theorem}

\begin{remark}
\label{bipartiteunique}
For each tree $T$ with $2m$ v\'ertices there is a spanning bipartite supergraph $F(T)$ with $1 + 2{m \choose 2}$ edges and a cute edge-colouring of $F(T)$ for which $T$ is the unique heterochromatic spanning tree.
\end{remark}

\begin{theorem}
\label{bipartitenice}
If $C$ is a nice edge-colouring of a complete bipartite graph $K_{m,m}$, then $K_{m,m}$ contains at least $(\frac{m}{2})(\frac{m+2}{2})$ heterocromatic spanning trees if $m$ is even and at least $(\frac{m+1}{2})^2$ heterochromatic spanning trees if $m$ is odd.
\end{theorem}


\begin{thebibliography}{}

\bibitem{AkAl06}
S. Akbari, A. Alipour, Multicolored trees in complete graphs, Journal of Graph Theory 54 (2006), 221 --- 232.

\bibitem{AN} 
J. Arocha, V. Neumann-Lara, Personal comunication.


\bibitem{BH96}
R. A. Brualdi, S. Hollingsworth, Multicolored trees in complete graphs, Journal of Combinatorial Theory, Series B 68 (1996), 310 --- 313.

\bibitem{BV01}
A. Bialostocki, W. Voxman, On the anti-Ramsey numbers for spanningtrees, Bull. Inst. Comb. Appl 32 (2001), 23 --- 26.


\bibitem{Ed70}
J. Edmonds, Submodular functions, matroids, and certain polyhedra. 1970 Combinatorial Structures and their Applications (Proc. Calgary Internat. Conf., Calgary, Alta., 1969) pp. 69 -- 87 Gordon and Breach, New York.

\bibitem{K73}
A. Kotzig, On certain vertex valuations of finite graphs, Util. Math., 4 (1973), 261 --- 290.

\bibitem{MR13}
J. J. Montellano-Ballesteros, E. Rivera-Campo, On the heterochromatic number of hypergraphs associated to geometric graphs an to matroids, Graphs and Combinatorics 29 (2013), 1517 --- 1522.

\bibitem{Rin63}
G. Ringel, Problem 25, Theory of Graphs and its Applications (Proc. Sympos. Smolenice 1963,
Nakl. CSAV, Praha, 1964), 162.

\bibitem{Ro67}
A. Rosa, On Certain Valuations of the Vertices of a Graph, Theory of Graphs (Proc. Internat.
Symposium, Rome, 1966), Gordon and Breach, N. Y. and Dunod Paris (1967), 349 --- 355.

\bibitem{Su06}
K. Suzuki, A necessary and sufficient condition for the existence of a heterochromatic spanning tree in a graph, Graphs and Combinatorics 22 (2006), 261 --- 269.




\end{thebibliography}
\end{document}